\documentclass[10pt]{article}

\font\smallit=cmti10

\usepackage{amssymb,latexsym,amsmath,epsfig,amsthm, url} 

\makeatletter

\renewcommand\section{\@startsection {section}{1}{\z@}
{-30pt \@plus -1ex \@minus -.2ex}
{2.3ex \@plus.2ex}
{\normalfont\normalsize\bfseries\boldmath}}

\renewcommand\subsection{\@startsection{subsection}{2}{\z@}
{-3.25ex\@plus -1ex \@minus -.2ex}
{1.5ex \@plus .2ex}
{\normalfont\normalsize\bfseries\boldmath}}

\renewcommand{\@seccntformat}[1]{\csname the#1\endcsname. }

\makeatother

\newtheorem{theorem}{Theorem}

\newtheorem{proposition}{Proposition}
\newtheorem{corollary}{Corollary}

\theoremstyle{definition}
\newtheorem{definition}{Definition}
\newtheorem{conjecture}{Conjecture}
\newtheorem{remark}{Remark}

\begin{document}
\begin{center}
\uppercase{\bf \boldmath On the divisibility of sums of Fibonacci numbers}
\vskip 20pt
{\bf Oisín Flynn-Connolly}\\
{\smallit Leiden Institute of Advanced Computer Science, Leiden University, Leiden, The Netherlands }\\
{\tt oisinflynnconnolly@gmail.com}
\end{center}
\vskip 20pt
\vskip 30pt

\centerline{\bf Abstract}
\noindent
We show that there are infinitely many odd integers $n$ such that the sum of the first $n$ nonzero Fibonacci numbers is divisible by $n$. This resolves a conjecture of Fatehizadeh and Yaqubi.

\pagestyle{myheadings}
\thispagestyle{empty}
\baselineskip=12.875pt
\vskip 30pt

\section{Introduction}
In a 2022 paper in the \emph{Journal of Integer Sequences} \cite{fat22}, the authors propose  the following conjecture on the divisibility of sums of Fibonacci numbers by their index.
\begin{conjecture}
\label{conj}
    There are infinitely many odd integers $n$ that divide the sum of the first $n$ nonzero
Fibonacci numbers.
\end{conjecture}
The primary purpose of this note is to provide a complete and affirmative resolution to this conjecture. We do so by constructing an explicit, infinite family of odd integers that satisfy the required divisibility property. Our main theorem (Theorem \ref{thm:main}) demonstrates that integers of the form $F_n$ with $n=2p$ or $n=4p$ for a specific class of primes $p$, fulfill the conditions of the conjecture. We additionally mildly strengthen and reprove several results of \cite{fat22}.

Our proof is based on studying those indices $n$ for which $F_n$ divides $\sum_{i=1}^{F_n} F_i$. We refer to the resulting sequence as the \emph{self-summable Fibonacci numbers}. To our knowledge, this sequence did not previously appear in the literature or in the Online Encyclopedia of Integer Sequences \cite{oeis}. It has since been added and now appears as A383021. The subsequence such that $F_n$ is odd is A381053.

We comment briefly on some related work. There is the aforementioned \cite{fat22}, which resolved the problem in the even case, by proving that 
$$
3 \cdot 2^{n+3} \mid \sum_{i=1}^{3 \cdot 2^{n+3}} F_i.
$$
K\u{r}\'i\u{z}ek and Somer \cite{michal24} extended many of their results to more general second-order recurrences. 

A similar sequence, consisting of those integers $n$ such that $n\mid F_n$, has been considered by several authors \cite{luca, smyth, somer} and has been called the \emph{self-Fibonacci numbers} and appears as sequence A023172 in the OEIS \cite{oeis}.

For clarity, we briefly recall notation. The $n^{th}$ Fibonacci number is denoted by $F_n$, where the Fibonacci sequence is taken to start at 1, i.e., $F_1 = 1$ and $F_2 = 1$ so that the sequence begins $1,1,2,3,5,\dots$. The Pisano period of $n$ is denoted $\pi(n)$; this is the period with which the sequence of Fibonacci numbers, taken modulo $n$, repeats. The sum of the first $n$ Fibonacci numbers is denoted $S_n$, i.e.,
$$
S_n = \sum_{i=1}^n F_i
$$
For more about the Fibonacci sequence, we refer the reader to \cite{Koshy}.

Finally, the remainder of this paper is structured as follows. In Section 2, we establish foundational results concerning sums of Fibonacci numbers, including the crucial identity $S_n = F_{n+2}-1$. We use this to show that the conjecture does not hold for prime indices. In Section 3, we introduce the concept of self-summable Fibonacci numbers and present our main theorem, which provides a constructive proof of Conjecture \ref{conj}.
\section{Sums of Fibonacci Numbers}
This section establishes foundational properties of sums of Fibonacci numbers that are essential for our main argument. We begin by proving the identity $S_n = F_{n+2}-1$, which provides a closed form for the sum. Using this result, we analyze the divisibility of $S_p$ for prime indices and establish a stronger version of \cite[Theorem 12]{fat22}: whereas the original theorem states that $p \nmid S_p$ for any odd prime  $p$, our result further determines the exact value of $ S_p \bmod p$.
 This analysis demonstrates that no prime index satisfies the conjecture's condition.
\begin{proposition}
\label{add}
    Let $n\in \mathbb N$. Then 
    $$
    S_n = F_{n+2}-1.
    $$
\end{proposition}
\begin{proof}
    The proof is a straightforward induction. For the base case,
    $$
    S_1 = 1 = 2 -1 = F_3 -1. 
    $$
    Then, by induction, suppose
    $$
    S_n = F_{n+2}-1.
    $$
    We have that 
    $$
    S_{n+1} = S_n+ F_{n+1} = F_{n+2} +F_{n+1} -1 = F_{n+3} -1
    $$
    as desired, where we have used $F_{n+2} +F_{n+1} = F_{n+3}$. 
\end{proof}
As mentioned before, we can now prove a stronger version of \cite[Theorem 12]{fat22} with a shorter proof using the Binet formula.
\begin{theorem}
\label{thm:primes}
    Let $p>5$ be a prime number. Then
    $$
    S_p \equiv \begin{cases}
        3 \bmod p &\mbox{if } p \equiv 1,4 \bmod 5.
        \\
        1 \bmod p &\mbox{if } p \equiv 2,3  \bmod 5.
    \end{cases}
    $$
    In particular, $S_p$ is not divisible by $p.$
\end{theorem}
\begin{proof}
    Recall the Binet formula which states that
    $$
    F_n = \frac{\phi^n - \overline{\phi}^n}{\sqrt{5}}
    $$
    where
    $$
    \phi = \frac{1+ \sqrt{5}}{2},  \qquad \overline{\phi} = \frac{1- \sqrt{5}}{2}.
    $$ 
    Now, rearranging
    $$
    F_{p+2} = \frac{\phi^{p+2} - \overline{\phi}^{p+2}}{\sqrt{5}}
    $$
    one can directly compute $F_{n+2}$ as
    \begin{equation*}
    \label{eq1}
    2^{p+1}F_{p+2} = \sum^{\frac{p+1}{2}}_{k=0} {p+2 \choose 2k+1} 5^k. 
     \end{equation*}
    It follows from Fermat's little theorem that 
    $$
    2^{p+1} \equiv 4 \bmod p.
    $$
    The binomial coefficients are divisible by $p$ except for $k =0, \frac{p-1}{2}$ and $\frac{p+1}{2}.$ 
    Therefore one has 
    $$
    4F_{p+2} \equiv {p+2\choose 1}5 + {p+2 \choose p}5^{\frac{p-1}{2}} + {p+2 \choose p+2} 5^{\frac{p+1}{2}} \bmod p
    $$
    Finally, this simplifies to 
    $$
    4 F_{p+2} \equiv 10 + \left( \frac{5}{p} \right) + 5 \left( \frac{5}{p} \right)  \bmod p
    $$
    using the identity
    $$
    5^{\frac{p-1}{2}}  \equiv \left( \frac{5}{p} \right) \bmod p
    $$
    where  $ \left( \frac{5}{p} \right)$ is the quadratic residue of 5 modulo $p$. Finally we know that 
    $$ \left( \frac{5}{p} \right) \equiv  \begin{cases}
        1 \bmod p &\mbox{if } p \equiv  1,4 \bmod 5
        \\
        -1 \bmod p &\mbox{if } p \equiv  2,3  \bmod 5
    \end{cases}$$
    The conclusion follows by applying Proposition \ref{add} and computing $F_{n+2}-1 = S_n$.
\end{proof}

The remaining cases can be computed by hand:
$$
S_3 \equiv  1 \bmod 3, \qquad S_5 \equiv  2 \bmod 5.
$$
\begin{remark}
    The natural approach to Conjecture \ref{conj} is to attempt to extend the strategy used for the proof of Theorem \ref{thm:primes} from primes to semiprime numbers $n=pq$, via the theorem of Lucas on prime divisors of binomial coefficients \cite{lucas}. However, this strategy rapidly becomes combinatorially intractable. One may also attempt to approach this problem via the Pisano periods of $p$ and $q$ and the work of \cite{wall}. This succeeds in producing sufficient (but not necessary) conditions on such primes $p, q$ in terms of simultaneous equations in modular arithmetic. However, it is not clear how to prove the infinitude of solutions to such equations using modern tools of number theory. Solving such linear Diophantine equations with prime solutions may become more tractable in the future thanks to recent advances in the field \cite{zhang}.
\end{remark}
\section{Resolution of Conjecture \ref{conj}}
\label{sec:self_summable}

In this section, we prove Conjecture \ref{conj}. Our strategy is to reframe the problem. Instead of directly finding odd integers $n$ that satisfy the conjecture's condition, we first introduce a new sequence of what we call \emph{self-summable Fibonacci numbers}. These are indices $k$ that satisfy the condition that the corresponding Fibonacci number $F_k$ divides the sum of the first $F_k$ Fibonacci numbers. Our main theorem then provides an explicit construction of numbers $n$ for which $F_n$ is both odd and self-summable.

The proof of this result relies on the periodicity of the Fibonacci sequence in modular arithmetic. We first establish a key proposition about the Pisano period, namely that for an even integer $n$, $\pi(F_n)$ divides $2n$. This allows us to reduce the nested self-summable condition to a much simpler modular congruence that can be solved directly.

\begin{definition}
        We say an integer $k$ is a \emph{self-summable Fibonacci number} if $F_k \mid \sum_{i=1}^{F_k} F_i$.
\end{definition}

\begin{remark}
    It follows from Proposition \ref{add} that the self-summable Fibonacci numbers are the integers $k$ such that $F_k \mid (F_{F_k+2} - 1)$. We note that the self-summable Fibonacci numbers are precisely the indices corresponding to the subsequence of OEIS A124456 \cite{oeis} that consists of Fibonacci numbers. The first terms are
    \[
    1, 2, 3, 12, 24, 34, 36, 46, 48, 60, 68, 72, 92, 94, 96, 106, \dots
    \]
    For some of the terms $k$ of this sequence, the corresponding Fibonacci number $F_k$ is odd. The first such terms are:
    \[
    1, 2, 34, 46, 68, 92, 94, 106, 166, 188, 212, 214, 226, 274, \dots
    \]
\end{remark}

We now show that one can explicitly describe a subsequence of the self-summable Fibonacci numbers where $F_n$ is odd.

\begin{theorem}
\label{thm:main}
Let $n = 2p$ or $n = 4p$, where $p$ is an odd prime such that 
\[
p \equiv 2 \bmod 3 \quad \text{and} \quad p \equiv \pm 2 \bmod 5.
\]
Then
\[
F_n \equiv 1 \bmod 2 \quad \text{and} \quad F_{F_n + 2} - 1 \equiv 0 \bmod{F_n}.
\]
\end{theorem}

The strategy behind this proof is essentially repeated use of the periodicity of Fibonacci numbers mod $m$, along with some explicit computation of relevant Pisano periods.

\begin{remark}
    In private correspondence, F. Luca has also pointed out to us that this theorem can also be proven via the identity 
    \[
    F_a - F_b = 
    \begin{cases}
        F_u L_v & \text{if } a \equiv b \bmod 4 \\
        F_v L_u & \text{if } a \equiv b+2 \bmod 4
    \end{cases}
    \]
    where $(u,v) = \left(\frac{a-b}{2}, \frac{a+b}{2}\right)$ and $L_v$ is the $v^{\text{th}}$ Lucas number.
\end{remark}

To prove this result, we first establish the following helpful proposition about the Pisano period of Fibonacci numbers, which may be of independent interest. Although this result seems likely to exist in the literature, we were unable to locate a reference, and the associated sequence does not appear in the OEIS at the time of writing.

\begin{proposition}
\label{prop:pisano_period}
   Let $n > 1$ and $2 \mid n$. Then $\pi(F_n) \mid 2n$.
\end{proposition}

\begin{proof}
    The fact that the Fibonacci numbers form a strong divisibility sequence and that $n \mid 2n$ implies that $F_n \mid F_{2n}$. It therefore suffices to establish that
    \[
    F_n \mid (F_{2n+1} - 1).
    \]
    By the addition rule for Fibonacci numbers, we can establish
    \[
    F_{2n+1} = F_{n+(n+1)} = F_n^2 + F_{n+1}^2.
    \]
    Therefore, it suffices to show that 
    \[
    F_n \mid (F_{n+1}^2 - 1).
    \]
    By Cassini's Identity, $F_{n+1}^2 - F_n F_{n+2} = (-1)^n$, we get
    \[
    F_{n+1}^2 \equiv (-1)^n \bmod{F_n}.
    \]
    When $n$ is even, $(-1)^n = 1$, and therefore
    \[
    F_{n+1}^2 \equiv 1 \bmod{F_n},
    \]
    which implies that $F_n \mid (F_{n+1}^2 - 1)$. This completes the proof.
\end{proof}

The following immediate corollary of Proposition \ref{prop:pisano_period} is likely well-known to experts but we could not locate it in the literature.

\begin{corollary}
    One has
    \[
    \lim_{N \to \infty} \min \left\{ \frac{\pi(n)}{n} : n \in \mathbb{N}, n < N \right\} = 0.
    \]
\end{corollary}

\begin{proof}
    Clearly,
    \[
    \min \left\{ \frac{\pi(n)}{n} : n \in \mathbb{N}, n < F_{M}+1 \right\} < \frac{\pi(F_M)}{F_M}.
    \]
    For an even $M$, by Proposition \ref{prop:pisano_period} we have $\pi(F_M) \mid 2M$, so
    \[
    \frac{\pi(F_M)}{F_M} \le \frac{2M}{F_M}.
    \]
    The right-hand side of the inequality goes to 0 as $M \to \infty$.
\end{proof}

We are now ready to prove our main result.

\begin{proof}[Proof of Theorem \ref{thm:main}]
    The conclusion that $F_n$ is odd for $n=2p$ or $n=4p$ (with $p \equiv 2 \bmod 3$) follows from the fact that $F_k$ is odd if and only if $k \not\equiv 0 \bmod 3$. In our case, $n$ is not divisible by 3, so $F_n$ is odd.

    To prove that $F_{F_n + 2} - 1 \equiv 0 \bmod{F_n}$, we shall use the periodicity of the Fibonacci numbers mod $F_n$. We consider first the case $n=2p$. By the definition of the Pisano period, $F_{F_n+2} \equiv F_k \bmod{F_n}$ for any $k$ such that $F_n+2 \equiv k \bmod{\pi(F_n)}$. In particular, as $\pi(F_n) \mid 2n$ by Proposition \ref{prop:pisano_period}, one may choose any $k$ such that
    \[
    F_n + 2 \equiv k \bmod{2n}.
    \]
    We are interested in the case where this congruence results in $F_k=1$, which occurs for $k=1$ or $k=2$. If we can show $F_n+2 \equiv 1 \text{ or } 2 \bmod{2n}$, then $F_{F_n+2}-1 \equiv F_{i} - 1 \equiv 1-1 \equiv0 \bmod{F_n}$ where $i\in \{1,2\}$. This requires us to establish that
    \begin{equation}
    \label{eq:main_congruence}
    F_n \equiv -1 \text{ or } 0 \bmod{2n}.
    \end{equation}
    Note that for $n=2p$, $\pi(2n) = \pi(4p)$. By the known divisibility relation for Pisano periods \cite{wall} and the fact $\pi(4) = 6$, we have $\pi(4p) = \text{lcm}(\pi(4), \pi(p)) = \text{lcm}(6, \pi(p))$.

    Next, we apply a result from Wall \cite{wall}, which states that when $p \equiv \pm 2 \bmod 5$, one has $\pi(p) \mid 2(p+1)$. The integer $p+1$ is divisible by 2 (since $p$ is odd) and by 3 (since $p \equiv 2 \bmod 3$). Thus $6 \mid (p+1)$, and we can conclude that $\pi(p) \mid 2(p+1)$. This implies $\pi(2n) = \text{lcm}(6, \pi(p))$ also divides $2(p+1)$.
    
    It therefore follows from the periodicity of the Fibonacci sequence mod $2n$ that $F_n \equiv F_{k'} \bmod{2n}$ for any $k' \equiv n \bmod{2(p+1)}$. Now,
    \[
    n \equiv 2p \equiv -2 \bmod{2(p+1)}.
    \]
    Thus, $F_n \equiv F_{-2} \bmod{2n}$. Since $F_{-2}=-1$, we have $F_n \equiv -1 \bmod{2n}$.
    
    This satisfies Equation (\ref{eq:main_congruence}), completing the proof for $n=2p$. A similar argument goes through for $n=4p$.
\end{proof}
Conjecture \ref{conj} follows as a simple corollary of Theorem \ref{thm:main}.
\begin{corollary}
    There are infinitely many odd integers $k$ such that $k$ divides the sum of the first $k$ Fibonacci numbers.
\end{corollary}

\begin{proof}
    By Dirichlet's theorem on primes in arithmetic progressions, there are infinitely many prime numbers satisfying the conditions of Theorem \ref{thm:main}. For each such prime $p$, Theorem~\ref{thm:main} ensures that, for $n=2p$ and $n=4p$, the corresponding Fibonacci number $F_n$ is an odd integer satisfying the required divisibility property, $F_n \mid \sum_{i=1}^{F_n} F_i$. This construction therefore yields an infinite set of distinct odd integers that satisfy the condition of the conjecture.
\end{proof}
\vskip 20pt\noindent {\bf Acknowledgement.} We thank Chase Ford and Joon Hyung Lee for useful conversations. We also thank Florian Luca, Yoshinosuke Hirakawa, Jos\'{e} Moreno-Fernandez and the anonymous reviewer for useful comments on the first version of this document. We thank Leiden University for its hospitality during the period when this paper was written.

\bibliographystyle{integers}
\bibliography{MyBib}

\bigskip

\end{document}